\documentclass[11pt,reqno]{amsart}

\usepackage{booktabs}                                 
\usepackage[english]{babel}
\usepackage{latexsym}
 \usepackage[utf8]{inputenc}
\usepackage{babel}
\usepackage{fancyhdr}
\usepackage{longtable}
\usepackage{mathrsfs}
\usepackage{geometry}
\usepackage{comment}
\usepackage{textcomp}
\usepackage{caption}
\usepackage{lmodern}
\usepackage{mathrsfs}
 \usepackage[T1]{fontenc}
\usepackage[all,cmtip]{xy}
\usepackage{epsfig}
 \usepackage{amsmath,amsfonts,amssymb,amsthm}
\usepackage{graphics}
\usepackage{graphicx}
\usepackage{verbatim}
\usepackage{dsfont}
\usepackage{enumerate}  
\usepackage{pdflscape}
\usepackage{longtable}
\usepackage{booktabs}
\usepackage{colortbl}
\usepackage[shortlabels]{enumitem}
\usepackage[breaklinks]{hyperref} 
\hypersetup{
	colorlinks,
	linkcolor={red},
	citecolor={blue},
	urlcolor={red}
}
\usepackage{stmaryrd}
\usepackage{multirow}

\usepackage[disable]{todonotes}

\newcommand{\C}{\mathbb{C}}

\newcommand{\QQ}{\mathbb{Q}}
\newcommand{\Q}{\mathbb{Q}}
\newcommand{\Z}{\mathbb{Z}}

\newcommand{\mQ}{\mathcal{Q}}

\newcommand{\mF}{\mathcal{F}}
\newcommand{\mG}{\mathcal{G}}

\newcommand{\PP}{\mathbb{P}}

\newcommand{\mU}{\mathcal{U}}

\newcommand{\of}{\mathcal{O}}

\newcommand{\W}{\bigwedge}

\DeclareMathOperator{\Ext}{Ext}

\DeclareMathOperator{\Sym}{Sym}

\DeclareMathOperator{\Gr}{Gr}
\DeclareMathOperator{\Fl}{Fl}
\DeclareMathOperator{\OGr}{OGr}
\DeclareMathOperator{\SGr}{SGr}

\DeclareMathOperator{\Bl}{Bl}

\newtheorem{thm}{Theorem}[section]
\newtheorem{corollary}[thm]{Corollary}

\newtheorem{proposition}[thm]{Proposition}

\newtheorem{conj}[thm]{Conjecture}

\newtheorem*{aim*}{Aim of this paper}

\newtheorem{method}[thm]{Method}

\theoremstyle{definition}
\newtheorem{definition}[thm]{Definition}

\makeatletter    
\def\l@subsection{\@tocline{1}{0,2pt}{2pc}{8mm}{\ \ }} 
\makeatletter    
\def\l@section{\@tocline{1}{0,2pt}{2pc}{8mm}{\ \ }}

\author{Enrico Fatighenti}
\address{Dipartimento di Matematica "Guido Castelnuovo"\\
Sapienza Universit\`a di Roma\\
 Piazzale Aldo Moro 5, 00185 Roma, Italy}
\email[E.~Fatighenti]{fatighenti@mat.uniroma1.it}

\title[Topics on Fano varieties of K3 type]{Topics on Fano varieties of K3 type}

\begin{document}
\begin{abstract}
This is a survey paper about a selection of recent results on the geometry of a special class of Fano varieties, which are called of K3 type. The focus is mostly Hodge-theoretical, with an eye towards the multiple connections between Fano varieties of K3 type and hyperk\"ahler geometry.
\end{abstract}
\maketitle

\thispagestyle{empty}

\section{Introduction}

Fano varieties of K3 type (FK3 for short), are a very special class of smooth projective varieties that lies at the crossroad of many different topics in algebraic geometry. In short, their peculiar name comes from the fact that they contain a Hodge-theoretical ``core'' that looks like the one of a K3 surface.
The archetypal example of a FK3 is the famous \emph{cubic fourfold}, i.e. the zero locus of a cubic polynomial in a five-dimensional projective space. Many interesting questions can be asked (and sometimes answered) about this variety, including, but not limiting to, the problem of its rationality and  the explicit construction of hyperk\"ahler manifolds linked to it. Even the techniques used vary wildly, from very classical Hodge-theoretical tools to their derived categorical counterpart, the study of stability conditions and so on.

Even more interestingly, very similar constructions can be performed starting from other known and lesser-known FK3. We hint in particular at the constructions of explicit examples of hyperk\"ahler manifolds (or IHS, i.e. Irreducible Holomorphic Symplectic), which are one of the most interesting yet elusive objects in algebraic geometry. Our not-so-hidden hope is in fact that the (easier) hunt for FK3 varieties will eventually lead us to new locally complete families of hyperk\"ahler manifolds still undiscovered.

This brief exposition will however only deal with the specific sub-topic of how to construct new examples of FK3 varieties and how to understand their Hodge structures. In particular, it will be a recollection of different results and constructions of the author and his collaborators, which have been obtained in a series of (quite long and verbose) papers in the last few years. Therefore, we will only briefly discuss the already known and world-famous examples, together with some results on their geometry. To be precise, in Section 2 we will touch upon some general facts on the classification of Fano varieties, the definition of FK3, and the extension to the categorical setting. We will briefly explain some recent and less recent techniques to attach families of hyperk\"ahler to families of FK3, the link with rationality questions and other topics, including the Hodge conjecture for FK3. We hope that the reader won't be disappointed by our synthesis of these topics, but anything more than this would fall way beyond the scope of this modest survey and would almost require a book on its own.

After this, in Section 3 we will jump straight into the action, discussing all the known examples of dimension four, including 64 new cases just recently discovered, and how they relate to the known examples from a birational and Hodge-theoretical viewpoint. We will then explore in Section 4 the higher-dimensional landscape, describing some Hodge-theoretical criteria to systematically produce FK3 varieties and some special interesting examples with a rich geometry. Finally, we will discuss a special set of correspondences (\emph{jumps and projections}) that will allow us to identify all those seemingly unrelated Hodge structures with one another.

We will finish this survey on a positive note, explaining what the future directions of our never-ending quest will be, and what the reader can expect us to find in the (hopefully!) not-so-distant future.

\subsection*{Notation} We work over the field of complex numbers. Every variety is assumed to be smooth and projective, unless otherwise specified. We will often work with varieties $X \subset G$ that can be described as zero loci $X=Z(\sigma)$, for $\sigma$ a general global section $\sigma \in H^0(G, \mF)$, with $\mF$ a (globally generated) vector bundle. We will write $X=(G, \mF)$ to indicate this. 
We will denote with $\Gr(k,n)$ the Grassmannian of $k$-dimensional subspaces in a $n$ dimensional complex vector space. It comes with two tautological vector bundles, i.e. $\mU$ of rank $k$ and $\mQ$ of rank $n-k$. Following our choice of convention, $\det (\mU^{\vee} )=\det(\mQ) \cong \of(1)$.

Similarly, if $E$ is a rank $r$ vector bundle over a variety $X$, we denote by $\PP_X(E)$ (or simply by $\PP(E)$) the projective bundle $\pi:\mathrm{Proj}(\Sym E^{\vee}) \rightarrow X$, i.e. we adopt the subspace notation. If we denote by $\of_{\PP(E)}(1)$ (or simply $\of(1)$) the relatively ample line bundle, this yields $H^0(\PP(E),\of_{\PP(E)}(1)) \cong H^0(X, E^{\vee})$. 

For products of varieties $X_1 \times X_2$, the expression $\mF_1 \boxtimes \mF_2$ will denote the tensor product between the pullbacks of $\mF_i$ via the natural projections.

\subsection*{Acknowledgments}
This survey paper has been written for a special edition of the Proceedings of Nottingham Algebraic Geometry Seminar, meant to celebrate one year of this weekly online meeting. I wish to thank the organizers: Al Kasprzyk, Livia Campo, and Johannes Hofscheier for their great effort in maintaining this throughout the whole pandemic crisis.

This survey is a summary of many papers written in the past few years: therefore I would like to thank first and foremost my collaborators on this project: Marcello Bernardara, Laurent Manivel, Giovanni Mongardi and Fabio Tanturri.

Throughout these years I had the pleasure of discussing these topics with many people, each of them providing me with their unique and precious point of view. Many of them gave also incredibly valuable feedback and comments on this set of notes. I would like to thank in particular Hamid Abban, Pieter Belmans, Vladimiro Benedetti, Valeria Bertini, Michele Bolognesi, Daniele Faenzi, Camilla Felisetti, Lie Fu, Atanas Iliev, Grzegorz and Micha{\l} Kapustka, Sasha Kuznetsov, Emanuele Macr\`i, Kieran O'Grady, Claudio Onorati, Miles Reid, Eleonora Romano and everyone else that I am guilty forgetting. I also thank the referee for their useful comments and remarks.

In these years I have been supported by various grants and projects, the last one being PRIN2017 "2017YRA3LK". The author is member of INDAM-GNSAGA.

\section{Generalities on Fano varieties and FK3}

\subsection{Fano varieties and their classification}
A Fano variety is a smooth projective variety $X$ such that its anticanonical bundle $-K_X$ is ample. One important property of smooth Fano varieties is their \emph{boundedness}, cf. \cite{kmm}: this in particular implies that for every dimension there exists a finite number of families of Fano varieties up to deformation. Of course, knowing that a classification is possible \emph{in theory} and actually having a classification are two completely separate matters. In fact, the full classification is known only in dimension up to three, and only (very) partial results are known in dimension four and beyond. We briefly recall that there are:
\begin{enumerate}
\item 1 Fano in dimension one;
\item 10 families of Fano in dimension two (Del Pezzo\footnote{There is a \emph{great debate} on whether the first letter of the surname should be capitalized or not. As Miles Reid correctly remarks in \cite{milesdp} the original spelling should be \emph{Pasquale del Pezzo}, since the preposition \emph{del} in a \emph{noble surname} needs not to be capitalized when preceded by the given name. However, as remarked in \cite{ciliberto}, there are different schools of thought on orthography when the surname is not preceded by the given name. Moreover, noble titles are not recognized anymore in Italy since 1948. Hence, in a republican spirit, we will write the first letter of the surname in upper-case.} surfaces);
\item 105 families of Fano in dimension three.
\end{enumerate}
The one-dimensional Fano is of course $\PP^1$. Del Pezzo surfaces are one of the most studied objects in algebraic geometry: from a birational viewpoint one has $\PP^1 \times \PP^1$, or $\PP^2$ blown up in (sufficiently general configurations of) up to 8 points. In particular, the anticanonical degree $-K_X^2$ of each Del Pezzo is 8 for $\PP^1 \times \PP^1$, and $9-k$ for $\Bl_{k}\PP^2$. Every family of Del Pezzo surfaces admits well-known biregular descriptions as well, i.e. by the means of equations (or more generally as $(G, \mF)$). To give an easy example, every Del Pezzo surface of degree 3, i.e. the blow up of $\PP^2$ in 6 points, can be realized as a smooth cubic surface $S=(\PP^3, \of(3))$, and vice versa. 

The situation is more complex in dimension three. Out of the 105 families, there are 17 of them which are prime, i.e. of Picard rank 1. They were originally studied by Fano himself and then later by Iskovskikh, \cite{isk}. Their biregular classification was later completed by Mukai \cite{mukai}. As for the remaining 88, the original classification completed by Mori and Mukai in \cite{morimukai} relied heavily on the birational tool of Mori theory. A biregular classification can be found in \cite{ccgk}, where every Fano 3-fold was described as either complete intersection in toric varieties or as zero loci in homogeneous varieties, and in \cite{dft}, where the entire classification was rewritten by considering only products of possibly weighted Grassmannians as ambient space. We refer to \cite{fanography} for an excellent overview.

From dimension four and onwards, only partial results are known. A good strategy to obtain some (weak) classification-type results is to stratify the set of families of Fano varieties by their \emph{index}, that is, the maximal integer $\iota_X$ such that $K_X$ is divisible by $\iota_X$ in Pic($X$). A basic fact is that, if $X$ is a smooth Fano of dimension $n$, the index $\iota_X$ is bounded by $n+1$. In general, fixing the dimension to $n$, Fano varieties are classified for the following values of the index and Picard rank $\rho_X$, and the lists can be found in \cite{ag5} and \cite{wis}:
\begin{enumerate}
\item $\iota_X=n+1$. $X \cong \PP^n$;
\item $\iota_X=n$. $X \cong \QQ^n$, i.e. a $n$-dimensional quadric hypersurface;
\item $\iota_X=n-1$. $X$ is called a \emph{Del  Pezzo manifold};
\item $\iota_X=n-2$ and $\rho_X=1$.  $X$ is called a \emph{Mukai manifold};
\item $\iota_X \geq \frac{n}{2}+1$ and $\rho_X >1$. $X \cong \PP^{\iota_X-1} \times \PP^{\iota_X-1}$ when the equality holds;
\item $\iota_X = \frac{n+1}{2}$ with $n$ odd and $\rho_X >1$. $X$ is called a \emph{Wi\'sniewski} manifold.
\end{enumerate}
We are interested in a special class of Fano varieties, which we call \emph{Fano varieties of K3 type} (FK3 for short). By construction these will have dimension at least 4, and index often just below the Wi\'sniewski threshold. We will now start with their definition and a first description of their properties.
\subsection{Fano varieties of K3 type}
We now define the main object of this survey, that is \emph{Fano varieties of K3 type} (FK3). We start with a series of Hodge-theoretical definitions.
\begin{definition}\cite[\S1]{r72} Let $H \cong \bigoplus_{p+q=k} H^{p,q}$ be a Hodge structure of weight $k$. We define the \emph{level} $\lambda(H)$ as \[ \lambda(H) := \textrm{max} \lbrace q-p \ | \ H^{p,q} \neq 0 \rbrace.\]
\end{definition}
The level of a Hodge structure measures its complexity. It is conventionally set to $-\infty$ if $H^{p,q}=0$ for all $p,q$ and it is equal to 0 if and only $H$ is central (or of \emph{Hodge-Tate type}) i.e. $H^{p,q} \neq 0$ only for $p=q$. We now introduce the key notion of K3-type Hodge structures.
\begin{definition} Let $H$ be a Hodge structure of weight $k$. We define $H$ to be of \emph{K3 type} if the following holds:
\begin{enumerate}
\item $\lambda(H)=2$;
\item $h^{\frac{k-2}{2}, \frac{k+2}{2}}=1$.
\end{enumerate}
\end{definition}
The above definition can be easily adapted to the case of $m$-Calabi-Yau Hodge structure, where the Hodge structure is required to be of level $m$ and with $h^{\frac{k-m}{2}, \frac{k+m}{2}}=1$. 

We point out that many (very different) varieties have Hodge structures of K3 type with this definition, including as well varieties of general type as the one constructed in \cite[Corollary 2.11]{surface}. We therefore need to restrict to our specific case of interest, as follows.

\begin{definition}\label{definition:fk3} Let $X$ be a smooth Fano variety. We say that $X$ is a Fano of \emph{K3 type} (FK3) if $H^*(X, \C)$ contains at least one sub-Hodge structure of K3 type, and in every $k$ there is at most one K3 structure.
\end{definition}
The above definition needs to be adapted if one considers Fano varieties of $m$-CY type. In fact, it is possible to find Fano varieties whose Hodge structure contains sub-structures of K3 and CY type at the same time, see e.g. \cite[Ex. S7]{eg2}. In those cases one should speak of Fano varieties of (pure) $m$-CY type or of mixed $(m_1, \ldots, m_s)$-CY type. However, for the purpose of this survey, we will only focus on the former, without adding this extra piece of notation.

Even though we will not focus on the cases $m\neq2$, they are equally interesting. For example, a first list of examples in the case $m=3$ (together with some extra condition) was produced by Iliev and Manivel in \cite{ilievmanivel}, together with a detailed study of their properties. This paper can be considered in fact as the beginning of this story.

We also point out that, from a purely Hodge-theoretical perspective, the Fano condition is in fact not necessary. It is of course possible to easily produce infinitely many examples of varieties of Kodaira dimension $-\infty$ of K3 type (for example, blowing up a FK3 or taking a projective bundle over it). Although unlikely, it is of course possible that there could be variety of general type as well of K3 type. However, it is customary to restrict to the Fano case, both for classification purposes and to use all the nice properties valid for this special class of varieties.

\subsection{Computing Hodge numbers}
Our definition of FK3 varieties is purely Hodge-theoretical, and at the same time does not pose any lattice-theoretical condition on the integral part of the Hodge structure. Hence, checking the Hodge numbers $h^{p,q}(X)$ is enough to determine the presence of a K3 structure. In general, for an arbitrary variety $X$, computing its Hodge numbers can be a difficult problem. However, the situation is (almost) algorithmic when $X$ can be written as $X=(G, \mF)$, with $G$ a homogeneous variety and $\mF$ a homogeneous, completely reducible, globally generated vector bundle. To begin with, if $X$ is cut by a general global section, $X$ is smooth and of the expected codimension, which is equal to the rank of $\mF$.

We briefly recall the strategy. Set rank$(\mF)=r$. For each $j \in \mathbb{N}$, we have the $j$-th exterior power of the conormal sequence
\begin{equation}
\label{wedgeKConormal}
0\rightarrow
\Sym^j \mF^\vee|_X \rightarrow
(\Sym^{j-1} \mF^\vee \otimes \Omega_G)|_X \rightarrow
\dotso \rightarrow
(\Sym^{j-k} \mF^\vee \otimes \Omega^k_G)|_X \rightarrow
\dotso \rightarrow
\Omega^j_G|_X \rightarrow
\Omega^j_X \rightarrow
0.
\end{equation}
To determine $h^i(\Omega^j_X)$, we can compute the dimensions of the cohomology groups of all the other terms in \eqref{wedgeKConormal}, split it into short exact sequences and use the induced long exact sequences in cohomology to get the result.

Each term $(\Sym^{j-k} \mF^\vee \otimes \Omega^k_G)|_X$ is in turn resolved by a Koszul complex
\begin{equation*}
0\rightarrow
\W^r \mF^\vee \otimes \Sym^{j-k} \mF^\vee \otimes \Omega^k_G \rightarrow
\dotso \rightarrow
\mF^\vee \otimes \Sym^{j-k} \mF^\vee \otimes \Omega^k_G \rightarrow
\Sym^{j-k} \mF^\vee \otimes \Omega^k_G.
\end{equation*}
In general, for any vector bundle $\mG$ we have a spectral sequence $E_*^{-q,p}$ with the first page  $E_1^{-q,p} = H^p(G, \mG \otimes \W^q \mF^{\vee})$ converging to $H^{p-q}(X, \mG|_X)$.

Since by hypothesis $\mF$ is completely reducible, then those terms are completely reducible as well: a decomposition can be found via suitable plethysms. The cohomology groups can be then obtained via the usual Borel--Weil--Bott Theorem, \cite{bott}, \cite{weyman}. Everything turns out to be particularly simple when $G=\Gr(k,n)$. In this case, most of these computations can be sped up by using computer algebra systems such as Macaulay2, \cite{M2}.

To see explicit examples of these computations we refer to \cite[3.3]{dft}, \cite[3.9.1]{eg2} or \cite[\S 3.2]{bft}.

\subsection{The triumvirate: Cubic fourfold, Gushel-Mukai fourfold and Debarre-Voisin 20-fold}

When speaking of FK3 varieties, there are three motivating examples which need to be mentioned as soon as possible: the \emph{cubic fourfold}, the \emph{Gushel-Mukai fourfold}\footnote{One can also construct a Gushel-Mukai 6-fold from a double cover of $\Gr(2,5)$.} and the \emph{Debarre-Voisin twentyfold}. They can be described, in our notation, as:
\begin{enumerate}
\item $(\PP^5, \of(3))$;
\item $(\Gr(2,5), \of(1) \oplus \of(2))$;
\item $(\Gr(3,10), \of(1)).$
\end{enumerate}

In other words, a cubic fourfold is constructed from a $f \in \Sym^3 V_6^{\vee}$, a Gushel-Mukai fourfold from a pair $(\varphi, \mu) \in \W^2 V_5^{\vee} \oplus \Sigma_{2,2}V_5^{\vee}$\footnote{The second space is exactly the quotient of $\Sym^2 \W^2 V_5^{\vee}$ by $V_5^{\vee}\otimes \det(V_5^{\vee})$, i.e. $H^0(\Gr(2,5), \of(2))$.}, and a Debarre-Voisin twentyfold from a $\sigma \in \W^3 V_{10}^{\vee}$. This representation-theoretical angle will be widely used throughout this survey.

 From a historical point of view, the cubic fourfold is the \emph{primus inter pares}. Being a hypersurface in a projective space, it is not hard to compute its Hodge numbers and check that $h^{1,3}=1$ and $h^{2,2}=21$, with its vanishing\footnote{The vanishing cohomology is defined in this case as the orthogonal to the classes from the ambient space, which in this case is only a power of the hyperplane class. In particular in this case it is the same as primitive cohomology.} component of rank 20\footnote{Exactly one more of the vanishing rank of the $H^2$ of a general K3 surface. Of course, not exactly a coincidence.}. Similar computations can be performed on the other two examples, which always happens to have a K3 structure with the same vanishing rank.

It is important to point out that every aspect that we will discuss in this survey (hyperk\"ahler geometry, derived category, rationality questions, etc.) are already present and often extremely well-developed for these three examples. There are in fact multiple and better surveys and papers where these questions are addressed, see e.g. \cite{huy,deb,dhov}\footnote{And many, many, many more.}, and we would like to invite the interested reader to consult them.

On the other hand, the purpose of this survey is to present new constructions of FK3, which have been more or less flourishing in the past few years. Therefore, we will just briefly recap some of the features of FK3 varieties we are mostly interested in, and which motivate our quest.
\subsection*{FK3: Hodge theory vs. derived categories}

The definition that we gave of FK3 is purely Hodge-theoretical. In fact, we are not imposing conditions on the full Hodge structure, but just on the  Hodge numbers. Our definition is loose on purpose, with the aim of including as many examples as possible. There are in fact other possible definitions of Fano varieties of K3 type, e.g. in the (stricter) categorical sense of Kuznetsov in \cite{kuzcy}. In general, a triangulated category $\mathcal{T}$ is said to be of K3 type if it has a Serre functor $S_{\mathcal{T}}$ with the property that $S_{\mathcal{T}}\cong [2]$. All the examples of the \emph{triumvirate} above have a Kuznetsov component $\mathcal{A}_X$ which is, in fact, a K3 subcategory: hence they are FK3 in all possible senses. It is important to point out that a FK3 in the categorical sense is as well a FK3 in our Hodge-theoretical sense, simply by taking Hochschild homology, as in, e.g. \cite[Lemma 3.5]{eg2}.

However, for most of the Fano which we will encounter the existence of a K3 subcategory is only conjectural, and in fact it can be quite difficult to prove, as we are going to see in the case of, e.g. K\"uchle (c5), or most of the interesting FK3 constructed in \cite{eg2}. In particular, the reader versed and interested in the categorical language can consider our definition of FK3 as a linear approximation of the categorical concept, and a hint on where to dig for new K3 categories, that is to say that our FK3 provide \emph{candidates} for finding K3 categories.

\subsection*{Link with hyperk\"ahler geometry} 
One of the main reasons (if not the most important one) for which we are interested in the study of FK3 varieties is the deep link with hyperk\"ahler geometry. A hyperk\"ahler manifold is a simply connected compact K\"ahler manifold\footnote{It is customary to give the definition in the complex geometry setting, since the general hyperk\"ahler is not projective, exactly as the case of the general K3 surface. However, we will only focus in the polarized case. The algebraically inclined reader is therefore encouraged to replace in their mind the words \emph{compact K\"ahler manifold} with \emph{smooth projective variety}.} with the additional property that $H^0(\Omega^2_X) \cong \C \sigma$, for a non degenerate holomorphic 2-form\footnote{This condition implies that a hyperk\"ahler manifold cannot be odd dimensional.} $\sigma$. Hyperk\"ahler manifolds in dimension 2 are K3 surfaces: in higher dimension however, there are just a few examples known. In particular, for any $n$, one\footnote{In particular, Beauville, \cite{beau}.} can construct two families of hyperk\"ahler manifolds, starting from a K3 surface and an abelian surface. These are often called K3$^{[n]}$ and $K_{n}(A)$\footnote{For the first one, one takes the Hilbert scheme of $n$ points on a K3 surface. In the second case, one takes the Hilbert scheme of $n+1$ points on an abelian surface $A$, and then the preimage of the 0 of the sum map $A^{[n+1]} \to A$. In both cases one deforms these varieties in order to get a locally complete example.}. There are then two sporadic examples, in dimension 6 and 10, constructed by O'Grady. No more examples which are not deformation equivalent to these are known. The search for new hyperk\"ahler manifolds, or more in general the study of their fascinating geometrical properties, is a very active subject in algebraic geometry. We refer to, e.g. \cite{debsur} for an overview.

The moral of the story is that, whenever one encounters a FK3, one should be able to construct (possibly many) examples of projective families of hyperk\"ahler, for example as moduli spaces of special objects in the given Fano variety. The archetypal example in this sense is the classical construction by Beauville and Donagi, \cite{bd}. Starting from a cubic fourfold $X$ as above, one can consider its variety of lines $F(X)$. This admits an easy description as $F(X)=(\Gr(2,6), \Sym^3 \mU^{\vee})$, and one can directly check that $F(X)$ is a hyperk\"ahler fourfold, moving in a locally complete family of deformations of varieties of K3$^{[n]}$ type. Other explicit constructions of hyperk\"ahler were carried out starting from both the Gushel-Mukai fourfold and the the Debarre-Voisin twentyfold, see e.g. \cite{deb} and \cite{dv}.

In fact, it was shown in \cite{km} that any moduli space of sheaves on a FK3 (in our Hodge-theoretical sense) carries a closed 2-form. The non-degeneracy of this 2-form is however in general difficult to check: on the other hand this happens to be true if the Fano is of K3 type in the categorical sense and if all the sheaves parameterized by this moduli space are objects of the subcategory $\mathcal{A}_X$\footnote{Once again, this shows how (Hodge-theoretical) FK3 can be considered an approximation for the stronger (categorical) FK3, and as a \emph{guiding light} if one is to look for hyperk\"ahler.}.

More recently, in \cite{blm+} and \cite{ppz} a theory of moduli spaces of Bridgeland semistable objects in the Kuznetsov component associated to a cubic fourfold and to a Gushel-Mukai fourfold was developed. In particular, infinite families of locally complete hyperk\"ahler manifolds of $K3^{[n]}$-type of different dimensions were constructed, and in \cite{lpz} the authors described a codimension one family\footnote{Which is a modular construction of a hyperk\"ahler compactification of the Lagrangian fibration constructed in \cite{lsv}.} in the moduli space of the 10-dimensional hyperk\"ahler example constructed by O'Grady in \cite{og99}.

We have little doubt that one could apply some generalization of these (and other) techniques to construct more families of hyperk\"ahler manifolds, which
in turn will be greatly helped by the production of more examples of FK3.

In all these techniques above, the link between the Fano and the hyperk\"ahler worlds goes always in the same direction, i.e. from \emph{Fano} to \emph{hyperk\"ahler}. We have recently seen an interesting way to \emph{reverse the arrow}. In \cite{fmos}, Flapan, Macr\`i, O'Grady and Sacc\`a explained a reverse process in which one starts with a hyperk\"ahler manifold X of K3$^{[n]}$-type and produces a corresponding FK3 variety arising as a connected component of the fixed locus of an antisymplectic involution on X. We expect their subsequent work to be extremely helpful for the production of new examples of FK3, further solidifying the bridge between the two worlds.

\subsection*{K3 structures and rationality of FK3}
The rationality of the cubic fourfold is a longstanding and intriguing open problem in algebraic geometry. Describing the many results on this topic falls way beyond the scope of this survey: we invite the interested reader to check, e.g. \cite{has, kuzsurvey, rs1, rs2, brs, a+}. In brief, the general cubic fourfold is expected to be irrational, but not a single example of irrational cubic fourfold is known. On the other hand, various examples of rational cubic fourfolds are known, but they all happen to be special in moduli. In fact, it is expected that rational cubic fourfolds belong to countable union of certain divisors $\mathcal{C}_d$ in the moduli space.
This \emph{speciality} condition has in fact everything to do with K3 structures. First of all, from the Hodge-theoretical viewpoint Hassett considered for certain $d$ called \emph{admissible} an associated K3 surface to a cubic fourfold $X \in \mathcal{C}_d$ with rank two sublattice $K \subset H^{2,2}(X,\Z)$ of discriminant $d$.

On the other hand, Kuznetsov constructed an associated K3 surface in the categorical sense to a cubic fourfold. In short, one defines the Kuznetsov component $\mathcal{A}_X$ as the right orthogonal to the exceptional collection given by $\lbrace \of_X, \of_X(1), \of_X(2) \rbrace$. This $\mathcal{A}_X$ is a K3 category, which in general is \emph{non-commutative}, i.e. not equivalent to $D^b(S)$, for $S$ a K3 surface. Kuznetsov's conjecture on the rationality of the cubic fourfold says that $X$ is rational exactly when $\mathcal{A}_X$ is in fact commutative (or \emph{geometric}). It was later shown by Addington and Thomas in \cite{at} that the geometricity of $\mathcal{A}_X$ is equivalent to $X$ belonging to some $ \mathcal{C}_d$, with $d$ admissible, linking the Hodge-theoretical and derived categorical worlds together.

A totally similar conjecture can be cooked up for Gushel-Mukai fourfolds, see \cite[Conjecture 3.12]{kp}. It is therefore natural to wonder if the same conjecture can be formulated for other FK3 fourfolds\footnote{For higher dimensional FK3 there should be no relations between the rationality of the varieties and the commutativity of the K3 structure, in any possible sense. Simplifying a lot, codimension two is what matters here, as in the classical case of the cubic threefold.}, even in the weaker Hodge-theoretical definition. As we are going to see later, in \cite{bfmt} we moved our first steps in this very direction.

\subsection*{Other topics: Chow rings and the Hodge conjecture for FK3}

Given a smooth projective variety $X$, let $A^i(X):=CH^i(X)_{\Q}$ denote the Chow groups of $X$ of cycles up to rational equivalence.
 The intersection product defines a ring structure on $A^\ast(X)=\bigoplus_i A^i(X)$, the Chow ring of $X$. In the case of K3 surfaces, the $\Q$-subalgebra $  R^\ast(X):=  \bigl\langle  A^1(X), c_j(X) \bigr\rangle \subset A^\ast(X) $ remarkably injects  into cohomology under the cycle class map, see \cite{BV}.
Motivated by the cases of K3 surfaces and abelian varieties, Beauville \cite{Beau3} conjectured that for certain special varieties, the Chow ring should admit a multiplicative splitting. To make concrete sense of Beauville's elusive ``splitting property conjecture'', Shen and Vial \cite{SV} have introduced the concept of \emph{multiplicative Chow--K\"unneth decomposition}. In \cite{lat20b} Laterveer conjectured that a FK3 should always admit such a \emph{multiplicative Chow--K\"unneth decomposition}. This has been verified for some of the FK3 constructed in \cite{eg2} see e.g. \cite{lat20b, latfm, latbol}. To this day, his work is still in progress, and we look forward to seeing more experimental verification of this conjecture. We invite the reader interested in Chow motives of FK3 varieties to check as well \cite{bp,b}.

Another topic which has been currently gaining momentum is the (integral) Hodge conjecture for FK3 varieties. No introduction should be needed on the Hodge conjecture, but we refer to \cite{voisinhodge} just in case. The integral Hodge conjecture (IHC) was already known to be true for cubic fourfolds, see \cite[Thm.18]{voisinhodge}. On the other hand, Mongardi and Ottem proved in \cite{mongardiottem} that for projective hyperk\"ahler of $K3^{[n]}$ or generalized Kummer type, the IHC holds for 1-cycles, a result which in turn can be used to re-prove the IHC for cubic fourfolds.

It is then natural to wonder if a similar result holds for other FK3 varieties. At the categorical level, some quite powerful progress has recently been made by Perry in \cite[Thm.1.1]{perry}. Starting from a Fano variety $X$ of (categorical) K3 type, he defined a version of the IHC for the Kuznetsov component $\mathcal{A}_X$\footnote{Which is a K3 subcategory.}. This is a statement variational in nature, i.e. asking for the K3 category to deform to a \emph{geometric} one, possibly up to a twist with a Brauer class, together with a certain integrality condition for a fixed Hodge class. His result implies in particular the (classical) IHC in degree 2 for the Gushel-Mukai fourfolds.

Finally, in \cite[Thm1.2]{bs} Benedetti and Song proved the IHC for the Debarre-Voisin twentyfold as well, this time using the Mongardi-Ottem result and hence the direct link with hyperk\"ahler geometry. In the same theorem, the authors of \emph{op.cit.} proved the IHC for another FK3, i.e. the Peskine variety $P$, which we will study in details later in this survey. It is interesting to remark that this $P$ is (at the moment) a FK3 only in our Hodge-theoretical sense\footnote{Although a categorification of the K3 structure has been long conjectured, for example by us in \cite[Conj. 21, Conj. 26]{bfm}.}. Therefore, we hope to be able in the future to prove the IHC for many other examples of (Hodge-theoretical) FK3, even when the existence of a K3 subcategory is yet to be determined.


\section{FK3 fourfolds}
In this section we will analyze all known FK3 fourfold, i.e. the most studied case. We will recap some classical and new results on the topic. 

\subsection{(Other) FK3 fourfolds of index $\iota_X>1$: Verra fourfold}

 Fano fourfolds of index $\iota_X>1$ have been classified, see e.g. \cite[Chapter 5]{ag5}, or \cite[4-6]{cgks} for an overview. In short, there are 35 such families: 1 of index 5, 1 of index 4, 6 of index 3 (out of which 5 with $\rho=1$) and 27 of index 2 (out of which 9 with $\rho=1$). We have already met the cubic fourfold and the Gushel-Mukai fourfold. There is another well-known example in literature of a FK3 fourfold, known as \emph{Verra fourfold}. This is a fourfold $V$ of index $\iota_V=2$ of Picard rank $\rho_V=2$. It can be described as the double cover of $\PP^2 \times \PP^2$ ramified along a smooth divisor of bi-degree $(2,2)$. Alternatively, $V$ can be described as a smooth quadratic section of the projective cone over $\PP^2 \times \PP^2$. 
 
 Verra fourfolds enjoy a rich geometry, see for example \cite{last}, \cite{ikkr}, \cite{laterverra}.
 The moduli spaces of Verra fourfolds is 19 dimensional (one less than the other mentioned FK3 examples). However, it turns out to be a particularly nice example. In fact, each of the two projections is a quadric fibration towards a base $\PP^2$, with ramification locus a (generically smooth) sextic curve. The double cover of each $\PP^2$ ramified over the said sextic is usually called an associated K3, $S_V$.
 
 This description implies that the K3 structure - both at the Hodge-theoretical and derived categorical level, cannot be identified with either one of the two prime FK3 fourfolds, nor coming directly from a K3 structure. In particular, (see \cite[5.1]{last}) from each projection one can deduce a semiorthogonal decomposition for the derived category $D^b(V)$, showing that it contains a piece equivalent to the Brauer-twisted derived category of the relevant associated K3 surface, $D^b(S_V, \beta)$.

Hyperk\"ahler fourfolds of K3$^{[2]}$-type can be constructed from the Verra fourfolds, for example starting from the Hilbert scheme of conics on $V$, \cite{ikkr}. As in the case of the cubic fourfold, the general Verra fourfold is conjectured to be irrational, with some special rational examples known. Moreover, the Hodge conjecture holds for them, \cite[2.3]{laterverra}.
 
 We give an easy description of the Verra fourfold as $(G, \mF)$, with $\mF$ and $G$ homogeneous, which was previously unknown to the best of our knowledge.
 \begin{proposition} A Verra fourfold $V$ can be described as $(\PP^2 \times \PP^2 \times \PP^9, \Lambda(0,0,1) \oplus \of(0,0,2))$, where $\Lambda$ is the bundle on $\PP^2 \times \PP^2$ of rank 8 defined as $\Lambda:= (V_3 \otimes \W^2 V_3) \otimes \of / \of(-1,-1)$ and then pulled back.
 \end{proposition}
 \begin{proof}
 We follow closely the tools of e.g. \cite[Fano 2-2]{dft}. In fact, by definition, a Verra fourfold is the double cover of $\PP^2 \times \PP^2$, hence can be described as $(\PP_{\PP^2 \times \PP^2}(\of \oplus \of(-1, -1)), \of_P(2))$. From the (pulled back) Euler sequence on one of the $\PP^2$ we get a sequence \[ 0 \to \of(-1, -1) \to \of^{\oplus 3}(-1,0) \to \mQ_{\PP(V_3)}(-1,0) \to 0, \]
 hence a sequence \[ 0 \to \of^{\oplus 3}(-1,0) \to  (V_3 \otimes \W^2 V_3) \otimes \of) \to \mQ_{\PP(\W^2 V_3)}^{\oplus 3} \to 0. \]
 Notice that the bundle $\Lambda$ is homogeneous: it can be in fact realized as an extension in $$\Ext^1_{\PP^2 \times \PP^2}( \mQ_{\PP(\W^2 V_3)}^{\oplus 3},  \mQ_{\PP(V_3)}(-1,0)) $$.
 \end{proof}
 
 The above description allows for a computation of the Hodge numbers using the already described methods. These numbers were already computed in \cite[2.5]{laterverra}, with a little typo in $h^{2,2}$. Another strategy would be to use weak Lefschetz theorem in all cohomology groups but the central one, and then compare with the Euler characteristic, which can be checked to be equal to 30 from a motivic computation $e(V)= 2 e(\PP^2 \times \PP^2) - e(V_t)$, with $V_t$ being the Fano 3-fold 2-6\footnote{In Mukai's notation, see \cite{fanography}.} (a \emph{Verra 3-fold}).
 \begin{corollary} The non-zero Hodge numbers for $V$ are $h^{0,0}=h^{4,4}=1$, $h^{1,1}=h^{3,3}=2$, $h^{1,3}=h^{3,1}=1$, $h^{2,2}=22$.
 \end{corollary}
 
 Among these 35 families there are no other FK3. We now discuss the remaining cases in dimension 4, i.e. the index 1 case.

\subsection{K\"uchle FK3 fourfolds of index $\iota_X=1$}
Fano fourfold of index $\iota_X=1$ are not classified (in fact, they are the first open case). There are several attempts at a classification, most notably using mirror symmetry, see e.g. \cite{mirror}, but we are quite far from achieving any final result on the topic. There are lists including hundreds of examples and classification in special subcases though, including e.g. \cite{batyrev}, \cite{kalashnikov}, \cite{ckp}.

A small but significant list was produced by K\"uchle in \cite{kuchle}, who was able to classify all Fano fourfolds of index $\iota_X=1$ that can be obtained as $X=(G, \mF)$, with $G=\Gr(k,n)$ and $\mF$ homogeneous, globally generated and completely reducible (not counting the already well-known complete intersections in $\PP^n$).

The result is a list consisting of 17 families of Fano fourfolds \footnote{The original list contained 18 families. However, recently Manivel proved in \cite{manivel} that the only two families with the same invariants can be identified.}, with a rich geometry and interesting properties. Three families out of the 17 are FK3. Using the original K\"uchle notation, they are:
\begin{itemize}
\item (c7): $(\Gr(3,8), \W^3 \mQ \oplus \of(1))$;
\item (d3): $(\Gr(5,10), (\W^2 \mU^{\vee})^{\oplus 2} \oplus \of(1))$;
\item (c5): $(\Gr(3,7), \W^2 \mU^{\vee} \oplus \W^3\mQ \oplus \of(1))$.
\end{itemize}

Fano c7 and d3 have Picard rank greater than one, where c5 is prime: most importantly, Kuznetsov in \cite{kuzpicard} showed that these two families are birational to much simpler varieties. Namely, we have the following result.
\begin{thm}[3.5 and 4.11 in \cite{kuzpicard}]
If $Z$ is a Fano of type $c7$, then $Z \cong \Bl_{v_2(\PP^2)}  X$, the blow up of a special cubic fourfold $X$ with center a Veronese surface. If $Z$ is a Fano $d3$, then $Z \cong \Bl_S (\PP^1)^4$, where the center of the blow up is a K3 surface obtained as the intersection of two linear sections. 
\end{thm} 

These identifications led to easy semiorthogonal decomposition of the derived category for both examples. Moreover, their rationality is obvious in one case, and reduces to the rationality of such a special cubic fourfold in the other case.

The case of Fano c5 is very different. It has Picard rank 1, which strongly points to a K3 structure different from the one of a cubic fourfold or a Gushel-Mukai. Moreover, by Hodge-theoretical reasons, for a generic c5 its K3 structure cannot be equivalent to the one of a (twisted) K3 surface.
Many interesting properties of this fourfold have been thoroughly studied by Kuznetsov in \cite{kuznetsovc5}.

It is expected that the c5 should be in many ways similar to both the cubic and the Gushel-Mukai fourfold, in particular with respect to hyperk\"ahler geometry and rationality properties. However, not a single rational example of c5 fourfold has been constructed. Also, it has long been conjectured that one should be able to construct a hyperk\"ahler of K3$^{[2]}$ type from the Hilbert scheme of twisted cubics contained in the  c5, but no substantial progress has been made yet.

There is another substantial difference between the cubic/GM duo and the $c5$, at the level of the derived category. The first two can be both considered as half-anticanonical sections of a fivefold ($\PP^5$ and a hyperplane section of $\Gr(2,5)$) where the derived category admits a rectangular Lefschetz decomposition. This, in turn, is enough to determine the structure of the derived categories of the corresponding fourfolds.
Of course the c5 fourfold can be interpreted as well as a half-anticanonical section of a Fano fivefold, which is simply $C_5:= (\Gr(3,7), \W^2 \mU^{\vee} \oplus \W^3\mQ )$. but the derived category is yet to be understood.

Kuznetsov studied the geometry of this fivefold in detail in \cite{kuznetsovc5}. He showed that the blow up $C_5$ with center the flag $\Fl(1,2,3)$ can be realized as the blow up of a hyperplane section of $\SGr(3,6)$ in a $\PP^1$-bundle over a DP$_6$. This is enough to show that the Chow motive of $C_5$ is of Lefschetz type, but not enough to describe completely its derived category.
The birational diagram restricts to the fourfold level, but with some caveat: in fact one gets that the blow up of the $c5$ fourfold $X_4$ in a DP$_6$ surface is birational to a singular fourfold, whose singular locus is a curve of genus 37. As a matter of fact, we have then a normal singular fourfold whose resolution of singularities is of K3 type. As we are going to see, this is definitely not an isolated phenomenon.

There is another difference between the $C_5$ fivefold and the other two: in fact both $\PP^5$ and a hyperplane section of $\Gr(2,5)$ are rigid, whereas $C_5$ has a 5-dimensional moduli space. With Claudio Onorati, we computed that $H^1(T_{C_5}|_X) \cong H^1(T_{C_5}) $, and this injects in the 25-dimensional infinitesimal deformation space  $H^1(T_X)$, which in fact splits as $H^1(T_X) \cong H^1(T_{C_5}) \oplus H^0(N_{X/C_5})$.
More precisely, we have the following
\begin{proposition}\cite{fo} Let $X, C$ be the K\"uchle $c5$ fourfold and fivefold as above. The image of the infinitesimal period map \[ \delta: H^1(T_X) \to Hom(H^{3,1}(X), H^{2,2}(X))\] coincides with the vanishing subspace $H^{2,2}_{v}(X)$ and the 5-dimensional kernel isomorphic to $H^1(T_C)$. Hence, the restricted period map \[ \overline{\delta}: H^1(T_X)/\ker \delta \cong  H^0(N_{X/C}) \to H^{2,2}_{v}(X)\]
is an isomorphism.
\end{proposition}

As a part of our ongoing project, we conjecture that starting from some special and possibly singular $X$ one can produce a map towards the moduli space of K3 surface of genus 12, with only finite fibers, hopefully giving us some intuition on the behavior of the period map in the general case.
\subsection{Other FK3 fourfolds in product of Grassmannians}
One of the most natural generalizations of K\"uchle's list is to consider Fano fourfolds obtained as zero loci of globally generated homogeneous vector bundles on product of Grassmannians (or more generally, type A flag varieties). If one is to pursue such classifications via e.g. representation-theoretical arguments as in K\"uchle's work, it is sensible to put extra conditions. For example, one might ask that, for $X=(G, \mF)$, the line bundle $K_G^{\vee} \otimes \det(\mF^{\vee})$ is already ample on $G$, hence before the restriction to $X$. Whenever $X$ satisfies this condition, we call $X$ \emph{strongly Fano}. Moreover, if $\mF$ is completely reducible, we call $X$ \emph{strongly reducible} (or SR) Fano.

SR Fano varieties are particularly interesting from a \emph{botanical} point of view: they can be classified - at least in principle - for example by the means of representation-theoretical arguments, computer algebra and clever boundedness tricks. Moreover, it is expected that they form a substantial percentage of the whole zoo of Fano varieties. For example, approximately 80\% of all Fano in dimension up to three admits a description as SR Fano. This is the point made in \cite{dft}, and we would expect a similar ratio in higher dimension as well.

As part of a long-term project, we plan to classify all SR Fano in dimension 4. This is a non-trivial endeavour, that will require some serious geometric understanding of the geometry of these varieties. As a first, partial, result, together with Bernardara, Manivel and Tanturri we produced a list of 64 families of (SR) FK3 fourfolds of index 1, of which 61 were previously unknown, at least in these terms\footnote{The three \emph{extra} are example (d3) in \cite{kuzpicard} and examples (b1) and (b2) in \cite{eg2}. We decided to include them, since we added some interesting geometrical information. We did not include in our list all the other known examples of FK3 fourfolds, including e.g. the cubic fourfold.}. Unsurprisingly, they all have Picard rank $\rho>1$.
Most of these families are in fact simple birational modifications of the above well-known examples. Nevertheless, their geometry sometimes happens to be quite interesting.
We subdivided our examples in four different disjoint groups, namely Fano of \textbf{C, GM, K3} or \textbf{R} origin.

Fano of \textbf{C} origin admits at least one birational map to a (either special or general) cubic fourfold. Similarly, Fano of \textbf{GM} origin  admit at least one birational map to a (either special or general) Gushel'-Mukai (GM) fourfold. We remark that the two sets are disjoint: in other words, no Fano in our list is birational to both a cubic and a GM fourfold.
From a rationality viewpoint, for Fano of C and GM origin, there is no surprise: we either recover the rationality of known special examples or we are unable to conclude.

Fano of \textbf{K3} origin are varieties which do not belong to any of the two previous groups, but  can be realized as blow ups of other Fano with central Hodge structure along a (sometimes non-minimal) K3 surface.
Varieties in each of these first three groups share many properties among each other, and we can often understand their K3 structure from the birational map in question. For almost every Fano of K3 origin, we can immediately conclude about their rationality. There are two families however, for which we cannot: their rationality is equivalent to the rationality of the Fano fourfolds of index 2 and degree 16, or the stable rationality of a smooth cubic threefold, which are open problems.

Finally, there is a small bunch of Fano of Picard rank 2 of \textbf{R} (or \emph{rogue}) origin that do not fall into any of the previous categories. They are in fact birational to singular fourfolds, and we can understand their K3 structure from their second projection (either a DP$_5$-fibration or a conic bundle). Summing up, we have:
\begin{enumerate}
    \item 17 Fano of \textbf{C} origin;
    \item 6 Fano of \textbf{GM} origin;
    \item 37 Fano of \textbf{K3} origin;
    \item 4 Fano of \textbf{R} origin.
\end{enumerate}
We invite the interested reader to consult \cite{bfmt} in order to understand in detail all these examples. Here, we only write down the four Fano of  \textbf{R} origin, which are the most mysterious (and unexpected) of the lot. They are (in the notations of the mentioned paper):
\begin{enumerate}
\item \textbf{R-61}(\#2-27-99-2): $(\PP^2 \times \Gr(2,6), \mU^{\vee}_{\Gr(2,6)}(0,1) \oplus \mQ_{\PP^2} \boxtimes \mU^{\vee}_{\Gr(2,6)})$;
\item \textbf{R-62}(\#2-33-129-2): $(\PP^3 \times \Gr(2,7),\of(0,1)^ {\oplus 2} \oplus \of(1,1) \oplus \mQ_{\PP^3}\boxtimes \mU^{\vee}_{\Gr(2,7)})$;
\item \textbf{R-63}(\#3-31-120-2-B): $(\PP^5 \times \PP^7,\mQ_{\PP^5}(0,1) \oplus \of(0,2) \oplus \of(2,0)^{\oplus 2})$;
\item \textbf{R-64}(\#3-25-90-2): $(\PP^4 \times \PP^6,\mQ_{\PP^4}(0,1) \oplus \of(3,0) \oplus \of(0,2))$.
\end{enumerate}
These four families of Fano behave all in a different way among each other. \textbf{R-61} can be partially understood by looking at the projection to $\PP^2$, which happens to be a DP$_5$-fibration. This (positively) settles the rationality question, and lets us describe the derived category as 6 exceptional objects together with the derived category of a certain non-Galois cover of $\PP^2$, which we conjecture to be the blow up of a K3 surface in three points.
\textbf{R-62} is quite mysterious: we understand the map to $\PP^2$ as a conic bundle with discriminant a quintic surface, but we only have a partial understanding of where the K3 structure (both at the Hodge and categorical level) comes from\footnote{In fact, the quintic surface is singular (nodal) in 16 points. If one takes the double cover of this quintic, one gets a surface of general type $D$, whose anti-invariant part in the second cohomology $H^2(D, \Z)^{-}$ is of K3 type and of rank 23, see \cite{huybrechtsQ}.}. Also, we cannot conclude anything on its rationality. In the last two examples, \textbf{R-63} and \textbf{R-64}, we understand pretty well the K3 structures: they can be constructed as conic bundles over (respectively) the complete intersection of two quadrics and a cubic threefold, with discriminant locus a K3 surface in both cases. In particular, their derived category contains a copy of the derived category of a K3 surface. However, we cannot say anything about their rationality yet, making them particularly nice candidates to study the generalized Kuznetsov conjecture.

We also remark that all these fourfolds admit birational projections towards certain singular fourfolds (e.g. certain singular complete intersections of three quadrics). It would therefore be interesting to study these fourfolds \emph{per se}, and understand if new K3 structure can be constructed from them.

\section{The higher dimensional landscape}
What happens when we try to understand FK3 varieties in dimension higher than four? It is quite likely that, at least with our level of technological advancement, any dream of classification remains hopeless (without restricting to very specific subcases). What is interesting however, is to produce new and geometrically interesting examples, especially if one is interested in the more \emph{hyperk\"ahler side} of life\footnote{We also remark that the rationality question becomes unrelated to the presence of a K3 structure whenever the dimension is higher than four.}. What becomes interesting then, is to understand \emph{a priori} how to produce such higher-dimensional FK3 examples, without embarking in a rather hopeless \emph{tour de force} in brute force. Luckily, we can pull something out of our bag of tricks. This is the content of \cite{eg1}, \cite{eg2} and \cite{bfm}.

\subsection{How to produce systematic FK3: reversing the Cayley trick} First of all, there is a systematic way to produce new FK3 varieties out of (some) of the old ones. This is the content of the well-known \emph{Cayley trick}, see e.g. \cite[Thm. 2.4]{kkll}. In its simplest version, it is nothing but a blow up lemma. Suppose in fact that we start with a smooth variety $X$ and a codimension 2 linear section $Z$. It is straightforward to check that the blow up $Y=\Bl_Z X$ can be described as the zero locus $Y=(X \times \PP^1, \of(1,1))$: in fact, if $Z$ is defined by the vanishing of $f,g$, one can simply take $y_0 f+y_1g$ as $(1,1)$-section. In particular, thanks to the well-known blow up Lemma for both Hodge structures and derived categories, if either $Z$ or $X$ has a K3 structure to begin with, then $Y$ will have such a structure as well.

More generally, suppose that we can write $Y$ as $Y= (\PP_X(E^{\vee}), \of_P(1))$, with $E$ of rank $r+1$ . Then the fibers of the projection $Y \to X$ will be generically $\PP^{r-1}$, with special fibers $\PP^{r}$ over the degeneracy locus $Z=(X, E)$. Once again, a variant of the blow up lemma ensures that the Hodge (and derived) structure of $Y$ can be read off from those of $X,Z$.

Suppose now that we can explicitly describe $Z$ (either K3 or FK3) as a $Z=(X, E)$. Starting from $Z$ we can \emph{reverse} the Cayley trick and associate to $Z$ another higher dimensional $Y$ defined as above with K3 structure, which sometimes happens to be Fano (hence, FK3), and sometimes admits a simple description. As an example, consider $Z$ a K3 surface of degree 8 in $\PP^5$, i.e. given by the complete intersection of three quadrics. In this case $E= \of(2)^{\oplus 3}$. Hence, to $Z$ we will associate an $Y$ given as $Y=(\PP^5 \times \PP^2, \of(2,1))$. This sixfold will be a FK3, with the K3 structure induced from the octic K3. Using this type of trick one can explain almost all\footnote{Possibly with the exception of M3--M5, which we explained using the formalism of Homological Projective Duality.} the \emph{systematic} FK3 in \cite{eg2}.

FK3 varieties produced in this way may seem ``boring'' from a Hodge-theoretical/derived categorical point of view, but can be interesting from a hyperk\"ahler point of view: in fact a few of explicit families of projective hyperk\"ahler described in \cite[Thm.1]{im19}, \cite[Prop.3.11]{eg2}, \cite[3.1.2, 3.1.3]{ben} can be constructed from FK3 obtained \emph{reversing the Cayley trick}. However, these families of hyperk\"ahler are specializations of well-known examples (respectively, the Hilbert square of a determinantal quartic K3, the variety of lines on a Pfaffian cubic fourfold, and the Hilbert square of  general K3 surfaces of degree 8 and 12). If one is to find (new) locally complete family of examples, we want to look for FK3 with genuinely new K3 structure.

\subsection{A Hodge-theoretical criterion}
The main difficulty here is, of course, being able to predict the Hodge structure of $X=(G, \mF)$ (or even only its level) only in function of the basic invariants of $G$ and $\mF$, such as the dimension of $G$, the degree and rank of $\mF$ and so on. In \cite{eg2} we tried to come up with such a criterion, starting from the analysis of the Griffiths' ring for complete intersections in Grassmannian from \cite{eg1}.  However imperfect\footnote{This method has to be interpreted as an indication of where a K3 structure might hide. In fact, as far as we know, one has to double check everytime that the suspected $X$ is in fact of K3 type.}, it turned out to be quite powerful in spotting previously unnoticed examples of FK3. The idea goes as follows:

\begin{method}\label{num} Let $Y$ be a smooth projective Fano variety of dimension $2t+1$ and index $\iota_Y$. 
Assume that $t$ divides $\iota_Y $ and that $\lambda(H^{2t+1}(Y))\leq 1$. Then a generic $ X \in | -\frac{1}{t} K_Y |$ is a Fano variety of K3 type, with the K3 type structure located in degree $2t$.
\end{method}

 The condition in \ref{num} is not necessary.  In fact notable exceptions are already present in \cite{eg1}, see e.g. Ex. S6, where the decomposition in irreducibles of the bundle that cuts the variety has no linear factor (albeit the variety has the correct ratio between dimension and index), and moreover two K3 sub-Hodge structures are present, in degree 6 and 8.

As it is stated, the method \ref{num} is also not an exact result. To turn it into a Theorem one would need to impose specific cohomological vanishing in each case. Therefore what could be gained in accuracy would be lost in terms of general applicability. It is important however, to notice that no counterexample to the condition in \ref{num} has been found yet.

What is good about the above method is that, whenever $G$ is a Grassmannian or another homogeneous variety, and $\mF$ is homogeneous, we can easily check the numerical relations needed. In fact, in this setting, there are few possible $Y$ to check in any dimension, giving us the possibility to hunt for meaningful examples.

\subsubsection{No other FK3 as complete intersections}
Of course the first attempt is to look for FK3 as complete intersections, in either projective space\footnote{The \emph{weighted} case was studied in \cite{ps}. Even in this case, no smooth Fano complete intersections were found.} or Grassmannians. Other than the known cases - the cubic and Gushel'-Mukai fourfolds and the Debarre-Voisin twentyfold, it was not a surprise to notice that no other complete intersection in $\PP^n$ or $\Gr(k,n)$ satisfy the condition in \ref{num}. On the other hand, there is a complete intersection in a Grassmannian which is of K3 type and lie outside of the range of \ref{num}. In fact, take $X=(\Gr(2,8), \of(1)^{4})$. This is a 8-fold of index 4, with a K3 structure in its middle cohomology, Picard rank 1 and a rank 19 vanishing subspace. A first explanation of the numerology can be found by simply realizing that the (classical) projective dual of the Grassmannian $\Gr(2,8)$ is a certain singular quartic in $\PP(\W^2 V_8^{\vee})$. If we intersect this quartic with the $\PP^3$ spanned by the four dimensional space of skew-symmetric two forms, we obtain a smooth (Pfaffian, hence in this case general) quartic in $\PP^3$, i.e. a K3 surface. This observation implies, thanks to the work of Segal and Thomas, \cite[Thm.2.9]{st}, that the derived category of the quartic K3 embeds into the one of the Fano 8-fold\footnote{Notice how this would follow from a general theory of homological projective duality for all $\Gr(2,n)$.}.

We believe that this (very peculiar) example could be the only exception: namely, in \cite{eg2} we formulated the following conjecture, which has not been proven yet.
\begin{conj} Let $X=X_{d_1, \ldots, d_c} \subset \Gr(k,n)$ be a Fano smooth complete intersection. Then $X$ is not of K3 type unless $$(\lbrace d_i \rbrace, k,n)=(\lbrace 3 \rbrace,1,6),(\lbrace 2,1\rbrace,2,5), (\lbrace 1,1,1,1\rbrace,2,8), (\lbrace 1 \rbrace, 3,10).$$ 
\end{conj}

\subsection{Sporadic FK3 in Fatighenti-Mongardi}
Our search of higher-dimensional FK3 in \cite{eg2} produced a long list of expected Fano which could be more or less explained by \emph{reversing the Cayley trick} starting from generic K3 in their Mukai model\footnote{Hence the name ``Mukai-type'' in the paper.}. More interestingly, we were able to produce a bunch of \emph{sporadic} examples, which deserved a case-by-case analysis. We list them here, and briefly recap some of their main features\footnote{We start from S2, since we called S1 the already mentioned codimension 4 linear section of $\Gr(2,8)$.}. They all have Picard rank $\rho=1$, with the only exception of S2.
\begin{enumerate}
\item S2: $(\OGr(3,8), \of(1))$. A Fano 8-fold of index $\iota=3$, with 1 K3 structure in weight 8.
\item S3: $(\SGr(3,9), \of(1))$. A Fano 14-fold of index $\iota=6$, with 1 K3 structure in weight 14\footnote{Independently discovered and studied by Iliev and Manivel in \cite{im19}.}.
\item S4: $(\SGr_2(3,8), \of(1))$. A Fano 8-fold of index $\iota=3$, with 1 K3 structure in weight 8.
\item S5: $(\Gr(2,9), \mQ^{\vee}(1) \oplus \of(1))$. A Fano 6-fold of index $\iota=2$, with 1 K3 structure in weight 6.
\item S6: $(\Gr(2,10), \mQ^{\vee}(1))$. A Fano 8-fold of index $\iota=3$, with 3 K3 structure in weights 6,8,10. By Lefschetz, a linear section of S6 is of K3 type as well\footnote{In the original paper, we decided to consider these as two separate examples, giving the latter the name S7. The reason for this was to highlight some important features of S7: in fact not only it has 2 K3 structure, but 1 3CY structure in weight 7 as well. As of today, is the only known prime Fano with this property.}.
\item S7: $(\Gr(k,10), \W^3 \mU^{\vee})$ for $k=4,5$. They\footnote{Contrary to the previous example, we grouped these Fano together since they share more similarities than differences.} are (respectively) a Fano 20-fold (15-fold) of index $\iota=7$ ($\iota=4$) with 7 (8) K3 structures in weights from 14 to 26 (from 8 to 22). By Lefschetz, 5 (3) linear sections of these varieties are still FK3.
\end{enumerate}

This list looks very diverse at a first sight, at least to the uneducated eye\footnote{As was ours when we first wrote about it.}. In fact, these Fano a priori share more differences than similarities. To begin with, they come in different dimensions, indexes, degrees and all that.  Moreover, they appear to be quite different even from a Hodge-theoretical perspective: in fact, as explained in the list above, in some case they come with one K3 structure, in some case with multiple ones, in different weights. A case-by-case analysis was in fact needed. In what follows, we will recap some of the most relevant features.

\subsubsection*{Fano S2} This Fano is, in some sense, the simplest of the lot. Notice first that $\OGr(3,8)$ itself can be described as the zero locus of a homogeneous vector bundle on $\Gr(3,8)$, this being obviously $\Sym^2 \mU^{\vee}$. In other words, one can describe the Fano S2 as the locus of three-dimensional subspaces $U_3 \subset V_8$ that are isotropic for a maximal rank quadratic form and for a given skew-symmetric 3-form.
The key in understanding this example lies, in fact, in the geometry of $\OGr(3,8)$ itself. In fact the latter can be considered as the projectivization $\PP_{\OGr^+}(\mU)$, where $\OGr^+:= \OGr^+(4,8)$ is one of the two connected component of the maximal orthogonal Grassmannian in its spinor embedding\footnote{It is worth noting that $ \OGr^+(4,8)$ is nothing but a 6-dimensional quadric.}. This also implies that the Picard rank of $\OGr(3,8)$ is 2, and the same by Lefschetz hyperplane theorem holds for S2, which is nothing but a linear section of this Grassmannian\footnote{In fact, is the only FK3 of the \emph{sporadic} group with Picard rank $\rho>1$.}. Also, the restriction of the Pl\"ucker line bundle from $\Gr(3,8)$ is the line bundle $\of(1,1)$, with respect to the basis of the Picard group given by the pullbacks of the hyperplane classes from the two spinor varieties.
By the Borel-Bott-Weil theorem, one has $H^0(\OGr(3,8), \of(1,1)) \cong H^0(\OGr^+(4,8), \W^3 \mU^{\vee})$. We can easily verify that the zero locus $W=(\OGr^+(4,8), \W^3 \mU^{\vee})$ is a (generic) K3 surface\footnote{One can also consider a second K3 surface of degree 12, namely in $\OGr^-(4,8)$. These two K3 surfaces where shown to be $\mathbb{L}$ equivalent but not isomorphic in \cite{ito}.} of degree 12 and genus 7. Equivalently, the restriction to S2 of the projection from $\OGr(3,8)$ to $\OGr^+(4,8)$ is generically a $\PP^2$ bundle, degenerating to a $\PP^3$ bundle exactly over $W$. These results imply that the Hodge structure of S2 (and in fact its derived category) comes from an actual K3 surface $W$\footnote{In fact, the vanishing subspace of $H^{4,4}$ has rank 19.}, which is as well the hyperk\"ahler associated to it.

\subsubsection*{Fano S3 and S4}
We collect here these two (apparently different) Fano, for reasons which will be clear in the next section. As in the previous case, we remark that the symplectic Grassmannian $\SGr(3,9)$ can be described as the zero locus of the homogeneous vector bundle $\W^2 \mU^{\vee}$ in $\Gr(3,9)$. Hence, the Fano S3 can be thought of the locus of three dimensional subspaces $U_3 \subset V_9$ which are isotropic for a maximal rank skew-symmetric 2-form\footnote{Since 9 is odd, maximal rank implies that the kernel of this two form is one dimensional.} and for a given skew-symmetric 3-form. Similarly, the bisymplectic Grassmannian $\SGr_2(3,8)$ can be described as the zero locus of the homogeneous vector bundle $(\W^2 \mU^{\vee})^{\oplus 2}$ in $\Gr(3,8)$, and the Fano S4 as the locus of three dimensional subspaces $U_3 \subset V_8$ which are isotropic for a pair of maximal rank skew-symmetric 2-form and for a given skew-symmetric 3-form. Surprisingly\footnote{Not really, a posteriori.}, they can be both linked to the Debarre-Voisin hyperk\"ahler fourfold $Z$. In fact in \cite[Prop.6]{im19} and \cite[Thm 3.26]{eg2} we were able to reconstruct this $Z$ as the zero locus of certain vector bundles on (birational modifications) of $\Gr(6,9)$ and $\Gr(6,8)$, and a variety parametrizing copies of $\SGr(3,6)$ and $(\PP^1)^3 \cong \SGr_2(3,6)$, see \cite{kuzpicard}, contained in S3 and S4.

\subsubsection*{Fano S5-S6-S7}
To describe S6 and S7, we start from a tri-vector $\sigma \in \W^3 V_{10}^{\vee}$. A Fano variety of type S6\footnote{Which is classically called \emph{congruence of lines}, see e.g. \cite{faenzi}.} will be then described as the locus of two dimensional subspaces $U_2 \subset V_{10}$ that annihilates this given $\sigma$. One can similarly define a Fano of type S7, from appropriately $k$-dimensional subspaces. In order to define S5 one starts instead from a tri-vector and a skew-symmetric two-vector over a 9-dimensional vector space: as we are going to see very soon, S5 can therefore be considered as similar to both S6/S7, and S3/S4. From S6 as well\footnote{And with little doubt to S5 and S7, with enough time and effort.} we were able to reconstruct the Debarre-Voisin hyperk\"ahler fourfold $Z$, as the space parametrizing certain special rational fourfold obtained as codimension 4 linear sections of $\Gr(2,6)$, see \cite[Prop.3.30]{eg2}.

\subsubsection{Peskine variety and sections of $E_6$}
One may wonder if considering rational homogeneous varieties from exceptional Lie groups one gets more interesting examples of FK3 varieties. However, this does not seem the case. There is a well-known case, where one takes as ambient variety the \emph{Cayley plane} $E_6/P_1$, and then take 6 linear sections: one gets a 10-dimensional Fano, of index 6, which is in fact a FK3. However, this can be explained using once again projective duality: in fact, as explained for example in \cite{im14}, the projective dual of the Cayley plane is the Cartan cubic: taking orthogonal linear sections one cuts the Cartan cubic in a smooth cubic fourfold. Hence, the K3 structure in this exceptional case is once again explained by a known one. This statement is conjecturally categorified in \cite[Conj. 1.2]{bks}.

Similarly, it seems that there should not be many more examples in low dimensions waiting to be discovered, even considering completely reducible bundles in arbitrary type.\footnote{We thank Pieter Belmans for this computation!}

The situation changes a lot if one considers - for example - degeneracy loci, which are basically uncharted territories. Already (briefly) in \cite{eg2} and much more in details in \cite{bfm} we considered the \emph{Peskine variety} $P$. To describe it, start (once again) with a tri-vector $\sigma \in \W^3 V_{10}^{\vee}$. For a $V_1 \subset V_{10}$ considered the skew-symmetric 2-form $\sigma(V_1, -, -)$ on $V_{10}/V_{1}$: this is generically of rank 8, but one can consider its Pfaffian locus, i.e. the locus where this form has rank 6 or less.\footnote{Equivalently, the zero locus of the 9 submaximal Pfaffians.} This is a prime Fano sixfold of index 3. In \cite{bfm}, we computed its Hodge structure, and checked that it has in fact 3 K3 structures, the first one being in weight 4\footnote{Notice how for a 6-fold, 3 is in fact the maximum number of possible K3 structure. In this sense, the Peskine variety is the ``best'' possible example of FK3.}. Once again, this Fano is linked to the Debarre-Voisin hyperk\"ahler fourfold\footnote{This tri-vector $\sigma \in \W^3 V_{10}^{\vee}$ is starting to look really suspect...}, which can be obtained as the space parametrizing the Palatini threefolds in $P$, see \cite{han}.
Also, as already remarked in the introduction, the integral Hodge conjecture holds for $P$, see \cite[Thm.1.2]{bs}.

\subsection{Jumps and projections: all these K3 structures are the same}
The ubiquity of the Debarre-Voisin hyperk\"ahler fourfold in this picture made us wonder if all these apparently different FK3 could somehow be related. In order to understand this, in \cite{bfm} we started from a very simple algebraic observation: if we take a skew symmetric tri-vector $\sigma \in \W^3 V_n$, and we choose a $v_0 \in V_n$ and therefore a decomposition $V_{n} \cong U \oplus v_0$, this induces a decomposition $\W^3 V_n \cong \W^3 U \oplus v_0 \wedge \W^2 U$. Projecting from $v_0$ we induce a rational map $\Gr(k, V_n) \dashrightarrow \Gr(k, U)$, which can be resolved blowing up the indeterminacy locus, which is isomorphic to $\Gr(k-1, U)$. 

If we now set $k=3$ and consider a skew tri-vector $\sigma$ as above, the zero locus $W$ of $\sigma$ is in fact a hyperplane section of $\Gr(3,V_n)$, i.e. $W=(\Gr(3,V_n), \of(1))$. Moreover, the above decomposition now reads $\sigma= \sigma' + v_0^* \wedge \omega$, with $\sigma \in \W^3 U^{\vee}$ and $\omega \in \W^2 U^{\vee}$. We can restrict the projection to $W$, and we still have a rational surjective map to $\Gr(3,U)$, which can be resolved by blowing up the zero locus of $\omega$, i.e. $\SGr(2, U)$. Hence, we end up with a map from $W$ to $\Gr(3, U)$ which (after blowing up), is a generically a $\PP^2$ bundle, with special fibers the whole of $\PP^3$ over the zero locus $X$ of $\sigma'$ and $\omega$, i.e. $X=(\Gr(3,U), \W^2 \mU^{\vee} \oplus \of(1))$, or equivalently $X=(\SGr(3,U), \of(1))$. We call this procedure the \textbf{projection} between $W$ and $X$\footnote{We can try to reiterate the procedure, jumping from $X$ to $(\SGr_2(3, U/v_1), \of(1))$, and even further. However the relative position of the projecting vector and the t-uples of 2-forms will make the degeneracy locus singular as soon as we do three jumps starting from an even-dimensional $V_n$, and two jumps from an odd-dimensional $V_n$.}. Thanks to \cite[Prop. 48]{bfm}, we can produce an isomorphism of integral Hodge structures in any weight between the Hodge structure of $\Bl_{\SGr(2,U)} W$ and the direct sum of the Hodge structures of $X$ and of $\SGr(3,U)$, counted with multiplicity. Of course we can use the classical blow up formula to relate this to the Hodge structure of $W$ as well.

There is a natural geometric construction that we can produce starting from $W$, for any $k$. Once again, start from $W=(\Gr(k, V_n), \of(1))$, given by the zero locus of a certain $\sigma$. Consider the projection $\pi: \Fl(k-1, k, V_n) \to \Gr(k,V_n)$, which is a $\PP^{k-1}$ bundle. We can now consider the zero locus of $\pi^*\sigma $ in  $\Fl(k-1, k, V_n)$, and the restriction of the projection $\overline{\pi}: \pi^*W \to W$, which is still a $\PP^{k-1}$-bundle.
On the other hand $\Fl(k-1, k, V_n)$ comes naturally equipped with another projection to $\Gr(k-1, V_n)$, which we can call $\varphi$, which is a $\PP^{n-k}$-bundle. In fact, we can think of $\Fl(k-1, k, V_n)$ as $\PP_{\Gr(k-1, V_n)}(\mQ(-1))$. If we restrict $\overline{\varphi}$ to $ \pi^*W $, we get a map towards $\Gr(k-1, V_n)$ which is generically surjective with fibers $\PP^{n-k-1}$, with special fibers $\PP^{n-k}$ over the zero locus $T=(\Gr(k-1, V_n), \mQ^{\vee}(1))$\footnote{Notice that $H^0(\Gr(k-1,n), \mQ^{\vee}(1)) \cong H^0(\Gr(k,n), \of(1)) \cong \W^k V_{n}^{\vee}$.}. We call this construction the \textbf{jump} from $W$ to $T$. Once again, we can use \cite[Prop. 48]{bfm} to understand the Hodge structure of $T$ in terms of these of $W$ and $\Gr(k-1, V_n)$\footnote{In \cite[Prop.49]{bfm} we developed a derived categorical version of Prop.48, which allows us to describe to some extent what happens to the derived category of $W$ after jumping and projecting. This leads us to a series of conjectures and reasonable expectations, see e.g. Conj.28, which are still not solved to this day, given the huge number of mutations to be performed.}.

Now take $k=3, n=10$. $W$ is the 20-fold Fano hypersurface originally considered by Debarre and Voisin. If we perform the first \emph{projection} we end up in $X=(\SGr(3,9), \of(1))$, i.e. \textbf{Fano S3} of the previous section. In this specific case, we can project a second time without hitting any singularity, and we end up in $Y=(\SGr_2(3,8), \of(1))$, i.e. \textbf{Fano S4} of the previous section\footnote{Notice that if we could project one time further, we would end up in $(\SGr_3(3, 7), \of(1))$, which is a K3 surface of degree 22 in its Mukai model. However, in order to project we have to degenerate the Fano to a singular one. In some sense, this can be considered a Hodge-theoretical version of the degeneration argument used in \cite{dv}.}. 

On the other hand, if you \emph{jump} from $W$ to $T$ we get $T=(\Gr(2,10), \mQ^{\vee}(1))$, i.e. \textbf{Fano S6} of the previous section. Moreover, projecting from $T$ one gets exactly \textbf{Fano S5}. Starting from $T$, one may also perform a slightly modified version of the jump, using as roof variety the Flag $\Fl(1,2,10)$. We obtain that a $\PP^1$-bundle over $T$ is in fact the blow up $\Bl_P \PP^9$, where $P$ is the \textbf{Peskine} variety previously considered. Finally, we can jump the other way around, and obtain in this way the Fano collected as \textbf{S7} in the previous section. We can visualize some of these maps using the following diagram, where $P$ denotes the Peskine 6-fold and $W$ the Debarre--Voisin 20-fold:
\hspace{-0.7cm}{\begin{equation*}
\xymatrix{
   & E  \ar@{}[dr]|{(2)} \ar@{^{(}->}[r]^{cdim 7} \ar[d]^{\PP^7} & q^*W  \ar[d]^{\PP^6} \ar[dr]^{\PP^2} & &\mathrm{Bl}_{S}W \ar@{}[dr]|{(3)} \ar[dl]_{bu} \ar[d]^{\PP^2} & F_1 \ar@{_{(}->}[l]_{cdim 3} \ar[d]^{\PP^3} & \\
   & T \ar@{^{(}->}[r]& \Gr(2,10) & W & \Gr(3,9) & X \ar@{_{(}->}[l] &  \\
 E' \ar@{}[dr]|{(1)} \ar[d]_{\PP^2} \ar@{^{(}->}[r]^{exc.div.} & q^*T \ar[d]^{bu} \ar[u]_{\PP^1} & & & & \mathrm{Bl}_{S_1} X \ar[u]^{bu} \ar@{}[dr]|{(4)} \ar[d]_{\PP^2} & F_2 \ar@{_{(}->}[l]_{cdim 3} \ar[d]^{\PP^3} \\
P \ar@{^{(}->}[r]_{cdim 3} & \PP^9& & & & \SGr(3,8) & Y,\ar@{_{(}->}[l]  
}
\end{equation*}}

From the above construction one can end up with the following theorem, which links up in a compact way most of the Fano considered in \cite{eg2}.

\begin{thm}{\cite[Thm.21]{bfm}}
The Hodge structure $K \cong H^{20}_{v}(W)$ is the minimal weight $2$ Hodge structure containing $H^{*-1,*+1}$ in the following Hodge structures:
\begin{itemize}
\item $H^{14}(X,\C)$,
\item $H^8(Y,\C)$,
\item $H^j(T,\C)$, for $j=6,8,10$,
\item $H^j(P,\C)$, for $j=4,6,8$.
\end{itemize}
Moreover, $H^{p,q}(\bullet)/ K = 0$ for $p\neq q$ for $\bullet$ either $Y_1$, $Y_2$, $T$ or $P$. 

Finally, if $Y$ is very general, then $K$ coincides with the vanishing cohomologies of all of the above cohomology groups for $Y_1$, $Y_2$, and for $P$ if $j=6$\footnote{Here ``vanishing'' means orthogonal to the sublattice generated by the third power of the hyperplane class and by the class of a Palatini 3-fold, see \cite{bs}.}.
\end{thm}

\section{Where is everybody?}
We are leaving the reader asking themselves the obvious questions: are there any other FK3 around? Can we find them? Will we be able to construct more hyperk\"ahler\footnote{The title of this section is an obvious reference to the \emph{Fermi paradox}, i.e. the apparent contradiction between the lack of evidence for extraterrestrial life and various high estimates for their probability. We are not comparing (yet) the quest for new hyperk\"ahler to the SETI program for the search of intelligent extraterrestrial life in the universe, but sometimes this task seems just as hard!}? Luckily, we can explain to the reader which roads we have not explored yet, and what we plan to do now.

First of all, one could always hope to find a new FK3 given as zero locus of a homogeneous, completely reducible vector bundle. After all, we don't have any arguments yet that involve crazily constructed bundles in incredibly high-dimensional Grassmannians. However, Benedetti in his thesis proved that any hyperk\"ahler constructed as the zero locus of an \emph{irreducible}\footnote{Or any fourfold, for a completely reducible, globally generated, homogeneous vector bundle.} homogeneous vector bundle over a (single) ordinary\footnote{This result was recently extended to the exceptional Grassmannian case by \cite{ex} for fourfolds, where the hypothesis of irreducibility is dropped as well.}
 Grassmannian must be either the variety of lines on a cubic fourfold or the Debarre-Voisin example, see \cite[Thm.2.1.20]{ben}. This of course does not rule out the possibility that a new higher-dimensional  locally complete family of hyperk\"ahler (and by analogy a family of prime FK3) can be found using some ad-hoc construction with more complicated vector bundles over a single Grassmannian, but does not make us optimistic either.
 
 However, Benedetti's result does not say anything about constructing hyperk\"ahler from vector bundles in \emph{product} of Grassmannians. In fact, we can find some of these examples in \cite{im19}, \cite{eg2}, \cite{ben}, in different dimensions. As we already discussed in the fourfold case, it is not hard to find examples of FK3 embedded in products, although we expect most of the K3 structures to come from already known ones. In particular, we expect the potential families of hyperk\"ahler to be non-locally complete (and very likely of K3$^{[n]}$ type).
 
 Even more interestingly, there is a direction the surface of which we have barely started scratching, i.e. degeneracy loci of morphisms of vector bundles (in their classical and orbital version, see e.g. \cite{bfmt20}). We have already seen how to construct a (prime!) FK3 as degeneracy locus, i.e. the Peskine variety $P$ of the previous section. Even from the hyperk\"ahler side degeneraci loci pop out quite frequently, see e.g. the famous construction of EPW sextics, \cite{epw}, i.e. special singular hypersurfaces constructed as degeneracy loci, whose double cover is a locally complete family of hyperk\"ahler fourfolds of $K3^{[n]}$ type.
 Very recently Benedetti, Manivel and Tanturri in \cite{bmt} constructed a family\footnote{Not locally complete - in fact it is of codimension one in the moduli space.} of hyperk\"ahler of \emph{generalized Kummer type} starting from a related\footnote{To construct the Peskine variety, one start from a skew tri-vector $\sigma \in \W^3 V_{10}^{\vee}$. To construct the Coble cubic, one start from a skew tri-vector $\sigma \in \W^3 V_{9}^{\vee}$.} degeneracy locus, i.e. the Coble cubic in $\PP^8$. It is not wishful thinking\footnote{Not totally, at least.} to imagine that more interesting examples -- both from FK3 and hyperk\"ahler perspective -- could and will be found as soon as our knowledge in this area expands.
 
  Finally, there is another direction which we are pursuing at the moment and that looks promising. We have already seen in the fourfold case that some K3 structure in FK3 seems to appear from \emph{singular} Fano. It would be interesting to reverse this perspective, and start looking directly for \emph{singular FK3}\footnote{There are several possible ways of generalizing the concept of FK3 to the singular case. However, one must exercise a certain degree of caution at this stage in order to avoid \emph{rookie mistakes}, and therefore we will keep the definition deliberately vague.} varieties. One could then try to understand the space of special subvarieties in these Fano, and see if they happen to be (degeneration of) family of hyperk\"ahler manifolds. However, as one should say in these cases, this is a story for another day.

\frenchspacing


\newcommand{\etalchar}[1]{$^{#1}$}

\end{document}